\documentclass[12pt]{amsart}

\usepackage{amsthm,amsmath,amssymb,amscd,amsfonts}

\theoremstyle{plain}
\newtheorem{theorem}{Theorem}[section]
\newtheorem{thm}[theorem]{Theorem}

\newtheorem{proposition}[theorem]{Proposition}
\newtheorem{prop}[theorem]{Proposition}

\newtheorem{def-thm}[theorem]{Definition-Theorem}
\newtheorem{lemma}[theorem]{Lemma}
\newtheorem{lem}[theorem]{Lemma}

\theoremstyle{definition}
\newtheorem{definition}[theorem]{Definition}
 \newtheorem{dfn}[theorem]{Definition}
 \newtheorem{notation}[theorem]{Notation}

\newtheorem{rem}[theorem]{Remark}

\newtheorem*{acknowledgement}{Acknowledgement}
\newtheorem{setup}[theorem]{Setup}

\renewcommand{\P}{\mathbb{P}}

\newcommand{\PP}{\mathbb{P}}

\newcommand{\OO}{{\mathcal O}}

\newcommand{\XX}{{\mathcal X}}

\newcommand{\FF}{{\mathcal F}}

\newcommand{\sss}{{\mathfrak s}}

\newcommand{\pr}{{\rm pr}}

\newcommand{\lra}{\longrightarrow}

\DeclareMathOperator{\rank}{rank}

\DeclareMathOperator{\Chow}{{Chow}}

\allowdisplaybreaks

\newcommand{\BC}{{\mathbb{C}}}
\newcommand{\BP}{{\mathbb{P}}}
\newcommand{\BQ}{{\mathbb{Q}}}

\newcommand{\CA}{{\mathcal{A}}}

\newcommand{\CI}{{\mathcal{I}}}
\newcommand{\CH}{{\mathcal{H}}}
\newcommand{\CL}{{\mathcal{L}}}
\newcommand{\CM}{{\mathcal{M}}}
\newcommand{\CN}{{\mathcal{N}}}
\newcommand{\CO}{{\mathcal{O}}}

\newcommand\ga{\alpha}

\newcommand\gc{\gamma}
\newcommand\del{\delta}
\newcommand\sg{\sigma}
\newcommand\w{\omega}
\newcommand\vph{\varphi}
\newcommand\ot{\otimes}
\newcommand\rd{\partial}
\newcommand\ol{\overline}
\newcommand\sm{\setminus}
\newcommand\sumn{\sum\nolimits}

\newcommand\tX{{\widetilde{X}}}
\newcommand\vx{\mbox{\boldmath $x$}_{\ge 1}}
\newcommand\vy{\mbox{\boldmath $y$}}

\begin{document}
\title[On uniformly effective birationality and the Shafarevich Conjecture]
{On uniformly effective birationality and the Shafarevich Conjecture
over curves}

\author{Gordon Heier}
\address{Department of Mathematics\\ University of Houston\\ 4800 Calhoun Road, Houston, TX 77204\\USA}
\email{heier@math.uh.edu}

\author{Shigeharu Takayama}
\address{Graduate School of Mathematical Sciences \\ University of Tokyo \\ 3-8-1 Komaba, Tokyo\\ 153-8914, Japan}

\email{taka@ms.u-tokyo.ac.jp}

\begin{abstract}
Let $B$ be a smooth projective curve of genus $g$, and $S \subset B$ be a finite subset of cardinality $s$. We give an effective upper bound on the number of deformation types of admissible families of canonically polarized manifolds of dimension $n$ with canonical volume $v$ over $B$ with prescribed degeneracy locus $S$. The effective bound only depends on the invariants $g, s, n$ and $v$. The key new ingredient which allows for this kind of result is a careful study of effective birationality for families of canonically polarized manifolds.\end{abstract}

\maketitle

\section{Introduction} \label{intro}

The origin of the problem addressed here is a conjecture due to Shafarevich,
which was proven by Parshin and Arakelov. 
The statement of the conjecture is as follows.
Let $B$ be a smooth projective curve of genus $g$, and $S \subset B$ 
be a finite subset of cardinality $s$. Then there are only finitely many isomorphism classes of smooth non-isotrivial 
families of curves of genus $g'$ greater than $1$ over $B \sm S$. 
Recall that a family of varieties is called {\it isotrivial} 
if generic fibers are isomorphic to each other.\par

Caporaso \cite{C} gave a uniform, but ineffective, bound on the number of
isomorphism classes of such families in terms of $g, g'$ and $s$,
and the first named author \cite{Heier_JMPA} gave an effective bound on 
that number, also depending on $g, g'$ and $s$.
The present work concerns the case of families of higher dimensional manifolds, 
while the base remains a curve.\par

We consider a smooth projective variety $X$ of dimension $n+1$, and
a surjective morphism $f : X \lra B$ such that 
$f$ is non-isotrivial and smooth outside $S$, and its smooth fibers are canonically polarized manifolds. We think of the smooth fibers $F$ as having either a fixed given {\it Hilbert polynomial} $h(m) = \chi(F,\CO_F(mK_F))$ or a fixed given {\it canonical volume} $v = K_F^n$.
We call such $f : X \lra B$ an {\it admissible family} over $(B, S)$ of 
canonically polarized manifolds with Hilbert polynomial $h$ or, respectively, with canonical volume $v$.
Our main result is the following. 

\begin{thm} \label{mt} 
Let $B$ be a smooth projective curve of genus $g$ and 
$S \subset B$ a finite subset with $s = \#S$. 
Then the number of deformation types of admissible families $f : X \lra B$ 
over $(B, S)$ of canonically polarized manifolds of dimension $n$ with 
canonical volume $v$ is bounded by an effective constant $C(g,s,n,v)$ depending only on 
$g, s, n$ and $v$. The number of deformation types of such admissible families with Hilbert polynomial $h$ is bounded by an effective constant $C(g,s,h)$ depending only on $g, s$ and $h$.
\end{thm}
The precise definition of {\it deformation type} is as follows.
\begin{definition}
(1) Let $T, \XX$ be irreducible quasi-projective varieties. 
A {\it deformation parametrized by $T$} of the admissible family $f:X\to B$ 
over $(B,S)$ is a holomorphic map $\FF:\XX\to B\times T$ such that 
$\FF:\FF^{-1}((B\setminus S)\times \{t_0\}) \to (B\setminus S)\times \{t_0\}$
is isomorphic to $f:X\setminus f^{-1}(S)\to B\setminus S$ for some $t_0\in T$ and $\FF:\FF^{-1}((B\setminus S)\times \{t\})\to (B\setminus S)\times \{t\}$
is a smooth family of canonically polarized compact manifolds 
for every $t\in T$.\par
(2) Two admissible families $f_1:X_1\to B,\ f_2:X_2\to B$ over $(B,S)$ 
are said to be {\it of the same deformation type} if there exist $T, \XX$ 
as above and $\FF$, a deformation parametrized by $T$ of $f_1$, such that 
$\FF:\FF^{-1}((B\setminus S)\times \{t_2\}) \to (B\setminus S)\times \{t_2\}$ is isomorphic to $f_2:X_2\setminus f_2^{-1}(S)\to B\setminus S$ for some 
$t_2\in T$. 
\end{definition}
Since the Hilbert polynomial $h(x)$ is of the form $(v/n!) x^n  +\ldots$, it is immediate that the bound by $C(g,s,n,v)$ is more general than that by 
$C(g,s,h)$. Nevertheless, we state an estimate by $C(g,s,h)$ for methodical and also traditional reasons. Note that, in the past, the focus was on boundedness in terms of  $g, s$ and $h$ (cp.\ \cite{BedVie}, \cite{KL}). The question of boundedness in terms of an (effective) constant $C(g,s,n,v)$ seemed to be mostly unaddressed.\par

The number $C(g,s,h)$ can be described as follows. Its geometric meaning will be explained in the main text.
Write $h(x) = \sum_{k=0}^n h_k x^k \in \BQ[x]$ with $h_n = K_F^n/n!$.
Let $m_0$ be the smallest integer which is not less than 
$(e+\frac12)n^{7/3} + \frac12 n^{5/3} + (e+\frac12)n^{4/3}
+ 3n + \frac12 n^{2/3} + 5$, where $e\approx 2.718$ is Euler's constant. 
Let $\mu_{0h} = \max \{k!m_0^k|h_k|;\ 0 \le k \le n\}$, and let
$$
	\ell_0^* = \sumn_{k=0}^n \gc_k \mu_{0h}^{(k+1)!}, 
$$
where $\gc_0 = 1, \gc_1=2$, 
$\gc_k = k^{k+1}\gc_{k-1}^{k+1} 
= k^{k+1} (k-1)^{k(k+1)} \ldots 3^{4 \cdot 5 \ldots k(k+1)}
(2^{3 \cdot 4 \ldots k(k+1)})^2$ for $k \ge 2$.
Let
\begin{equation*} 
\begin{aligned}
\del(m) 
& = (n(2g-2+s) + s) \cdot m \cdot (m^nK_F^n + 1) \cdot h(m) 
	\ \ \text{for $m=m_0$ or $m_0\ell_0^*$}, \\
d(k) 
& = \del(m_0k) + 2g \cdot k \cdot h(m_0k) 
	\ \ \text{for $k=1$ or $\ell_0^*$}, \\
N & = d(1) + (1-g)h(m_0) - 1,  \\
d & = d(1) (\ell_0^*+1)^{h(m_0)-n-1} 
		+ (h(m_0)-n-1)(d(\ell_0^*) + 2g) (\ell_0^*+1)^{h(m_0)-n-2}.
\end{aligned}
\end{equation*} 
The above $N$ and $d$ depend only on $g, s$ and $h$. We remark that $2g-2+s>0$ by \cite[Theorem 1.4(a)]{BedVie}. Then we set
$$
C(g,s,h) = \sum_{\nu=1}^d{(M+1)p_\nu \choose M}
^{(M+1)\left(p_\nu {p_\nu+n\choose n+1 }+{p_\nu+n \choose n}\right)},
$$
where $M=(N+1)(g+2)-1$ and $p_\nu = (n+1)(2g+1)\nu$.\par
To obtain the constant $C(g,s,n,v)$ from the numbers defined above,
it is enough to bound all the coefficients of a Hilbert polynomial 
$h(x) = (v/n!) x^n + \ldots$ in terms of $n$ and $v$ effectively,
as in the following Proposition \ref{coeff_bound_1}.
Then the above $\mu_{0h}, \ell_0^*, \del(m), d(k), N = N(g,s,h), d = d(g,s,h)$, and hence $C(g,s,h)$, are bounded above by effective numbers 
depending only on $g, s, n$ and $v$.
Thus, $C(g,s,h)$ is converted to $C(g,s,n,v)$. We will leave making them more explicit to the reader. Note that Proposition \ref{coeff_bound_1} will be proven in Subsection \ref{sec_coeff_bound}, after being restated as Proposition \ref{coeff_bound}.
\begin{prop} \label{coeff_bound_1}
Let $F$ be a canonically polarized manifold of dimension $n$,
and let $\chi(F,\CO_F(xK_F)) = \sum_{i=n,\ldots,1,0} h_i x^i \in \BQ[x]$
be the Hilbert polynomial. 
Then $h_n = K_F^n/n!$ and
$$
	|h_{n-k}| <  n! a_1 \cdots a_n m_n^k(1+m_n)^{nk} K_F^n 
$$
for $k = 0,1,\ldots,n$, where $m_n = 1 + \frac12 (n+1)(n+2)$ and 
$a_p = 2^{p(p+3)/2-2}/p!$ for $p \ge 1$.	
\end{prop}
To put the present work in the proper perspective, we remark that it has been inspired by the earlier paper \cite{Heier_Crelle}, which contains the following Theorem \ref{crelle_thm}. In particular, this earlier paper developed a new method to identify a given family with an embedded projective model in a way that made unnecessary the technically challenging iterated use of Chow or Hilbert varieties, which was the hallmark of the earlier approaches such as \cite{P}, \cite{Heier_JMPA}. In the statement of Theorem \ref{crelle_thm}, the symbol $F$ denotes the fiber over some fixed base point in $B\setminus S$.
\begin{theorem}\label{crelle_thm}
Let $\widetilde d, p$ be positive integers, $\widetilde N =(g+2)(p+1)-1$,
and let $m_0 = O(n^{7/3})$ be the integer mentioned above.
Then the number of deformation types of admissible families $f:X\to B$ over 
$(B,S)$ with ``moving intersection numbers'' satisfying
$\big(m_0K_X+(\tfrac {m_0} 2 (g+1)+2g+1)F\big)^{[n+1]} = \widetilde d$ and 
$\big(m_0K_X+\tfrac {m_0} 2 (g+1) F\big)^{[n+1]}+n \le p$ is no more than
\begin{equation*}
 {(\widetilde N+1)\widetilde d\choose \widetilde N}
 ^{(\widetilde N+1) \left(\widetilde d {\widetilde d+n\choose n+1 }
				+ {\widetilde d+n \choose n}\right)}.
\end{equation*} 
\end{theorem}

When $K_X$ is nef, the effective bound in Theorem \ref{crelle_thm} can be estimated from above in terms of $(g,n,K_F^n,K_X^{n+1})$ as explained in \cite[Remark 2.7, Lemma 2.8]{Heier_Crelle}. It would then clearly be desirable to directly bound $K_X^{n+1}$ in terms of $(g,s,n,K_F^n)$ or at least $(g,s,h)$. In the case of $1$-dimensional fibers, this is done in \cite[Proposition 1]{P}. However, in the higher dimensional situation, it does not seem to be known how to accomplish this (cp. \cite{LTZ}). The present paper circumvents this problem by using an embedding that is better suited to the specific geometric situation at hand.\par

As for the history of our problem, recall that Bedulev and Viehweg proved the following in \cite{BedVie} under the assumption that the Minimal Model Conjecture holds. Let $f:X\to B$ be an admissible family over $(B,S)$ such that one (and therefore every) smooth fiber has Hilbert polynomial $h$. Then the number of deformation types of admissible families over $(B,S)$ whose smooth fibers also have Hilbert polynomial $h$ is finite. Kov\'acs and Lieblich \cite{KL} then showed that this number can uniformly, but ineffectively, be bounded by a constant depending only on $g$, $s$ and $h$. \par

There are two kinds of effective arguments needed to obtain Theorem \ref{mt}. The first is a pluricanonical birational embedding of $X$ into a projective space
$\BP^N$ with effective bounds on $N$ and on the degree of $X$ in $\BP^N$. This part is the key new technical result obtained in this paper. As one may suspect, the degree bound of $f_*\CO_X(mK_{X/B})$ due to Bedulev-Viehweg \cite[Theorem 1.4]{BedVie} is important in our argument. Another important input comes from a relation between Hilbert polynomials and Castelnuovo-Mumford type regularity. We will use not only the degree, the coefficients, and the values of the Hilbert polynomial $h$, but also the {\it length} of the binomial
sum expansion. However, no results from \cite{KL} will be used.\par

The second argument consists of effectively embedding admissible families in a projective space such that the number of deformation types is bounded by the number of irreducible components of a certain Chow variety $\text{Chow}_{n+1,d}'(W)$ of $(n+1)$-dimensional varieties of degree $d$ which are contained in a certain projective variety $W$. This part of the argument is similar to the corresponding part in \cite{Heier_Crelle}.

We work over the complex number field $\BC$.

\begin{acknowledgement}
The second named author would like to thank Professor Yoichi Miyaoka
for a number of inspiring discussions.
\end{acknowledgement}

\section{Uniformly effective birationality}\label{ueb}
We consider, as in the Introduction, 
admissible families $f : X \lra B$ over $(B, S)$ of canonically polarized
manifolds with Hilbert polynomial $h$.
We fix $B, S$ and $h$; in particular, $g$ and $s$ are also fixed. 
The following theorem is the key new effective boundedness result in this work. 

\begin{thm} \label{effbir}
For given $B, S$ and $h$, there exist effective positive integers $N = N(g, s, h)$ and 
$d = d(g,s,h)$ depending only on $g, s$ and $h$ with the following properties.
For any admissible family $f : X \lra B$ over $(B, S)$ of canonically 
polarized manifolds with Hilbert polynomial $h$,
there exists a rational map $\Phi : X \dasharrow \BP^N$, which is birational onto its image and gives a regular embedding on $X \setminus f^{-1}(S)$ such that the degree of the image of $X$ is bounded by $d$, i.e., $\deg \Phi(X) \le d$.
In the case $g \ge 2$, one can take $\Phi$ to be a pluricanonical map
 $\Phi_{|m_0K_X|}$ with $m_0$ as defined in Notation \ref{numbers}(1) depending only on $n$
and with possibly different effective integers $N' = N'(g, s, h)$ and $d' = d'(g,s,h)$.
\end{thm}

The effective integers $N, d$ will be given in Definition \ref{def_const}. One of the key ingredients is the following invariant of Hilbert polynomials.

\begin{dfn}
Let $F \subset \BP$ be a closed subscheme of dimension $n$
in a projective space $\BP$. We denote by $\CO(1)$ the ample line bundle on $F$ which is the restriction of $\CO_\PP(1)$. Let $P(x) \in \BQ[x]$ be the Hilbert polynomial of $F$ with respect to 
$\CO(1)$, i.e., $P(m) = \chi(F, \CO_F(m))$ holds for all 
sufficiently large integers $m$ (\cite[Theorem I.7.5]{Hartshorne}).
By a theorem of Gotzmann \cite{G} (\cite[Theorem 1.8.35]{PAGI}, \cite[Theorem 4.3.2]{BH}), 
there exists a unique finite sequence of integers 
$a_1 \ge a_2 \ge \ldots \ge a_\ell \ge 0$ such that
$$
   P(x) = \binom{x+a_1}{a_1} + \binom{x+a_2-1}{a_2} + \ldots 
			+ \binom{x+a_\ell-(\ell-1)}{a_\ell}.
$$
We will refer to the integer $\ell$ as the {\it length} of the binomial sum 
expansion of $P(x)$.
\end{dfn}

Recall that 
$\binom{x}{a} = \frac1{a!} x(x-1) \ldots (x-a+1)$,
which is a polynomial of degree $a$ for a positive integer $a$, and 
$\binom{x}{0} = 1$ (\cite[Proposition I.7.3]{Hartshorne}).
If we write $P(x) = p_nx^n + p_{n-1}x^{n-1} + \ldots + p_1x + p_0$
with $p_i \in \BQ$, we can write $a_1, \ldots, a_\ell$ and $\ell$
in terms of $p_n, \ldots, p_0$ and $n$ in recursive relations.
For example, the sequence starts with $a_j = n$ for $1 \le j \le n!p_n$, 
and $a_j < n$ for $j > n!p_n$.
We can also give an effective bound of $\ell$ in terms of $p_n, \ldots, p_0$ and $n$,
see Lemma \ref{ell}.

We shall use the following effective positive integers.

\begin{notation} \label{numbers}
(1) Let $m_0$ be the smallest integer which is not less than 
$(e+\frac12)n^{7/3} + \frac12 n^{5/3} + (e+\frac12)n^{4/3}
+ 3n + \frac12 n^{2/3} + 5$, where $e\approx 2.718$ is Euler's constant. From \cite{heier_doc_math} we know that, for any $m \ge m_0$, $|m K_Y|$ is very ample for any compact complex manifold $Y$ of dimension $n$ with ample $K_Y$. Earlier such bounds  were given by Demailly \cite{D} ($m_0 = O(n^n)$), 
and by Angehrn-Siu \cite{AS} ($m_0 = O(n^3)$).\par

(2) 
Since $|m_0K_F|$ is very ample for any smooth fiber $F$ of $f : X \lra B$,
there exists a polynomial $P(x) \in \BQ[x]$ of degree $n$ such that
$P(m) = \chi(F, \CO_F(m_0mK_F)) = h(m_0m)$
for all sufficiently large integers $m$. In fact, if we write $P(x) = \sum_{k=0}^n p_k x^k$ and $h(x) = \sum_{k=0}^n h_k x^k$,
then $p_k = m_0^kh_k$ for $0 \le k \le n$.
Let $\ell_0$ be the length of the binomial sum expansion of the Hilbert 
polynomial $P(x) = h(m_0x)$.
Although $\ell_0$ is written in terms of $p_n, \ldots, p_0$ and $n$, it is not easy
to write it in a simple form.
Instead, we give an effective bound in Lemma \ref{ell}:\
$$
	\ell_0 \le \sumn_{k=0}^n \gc_k \mu_{0h}^{(k+1)!}=: \ell_0^*,
$$
where $\gc_0 = 1, \gc_1=2$, 
$\gc_k = k^{k+1}\gc_{k-1}^{k+1} 
= k^{k+1} (k-1)^{k(k+1)} \ldots 3^{4 \cdot 5 \ldots k(k+1)}
(2^{3 \cdot 4 \ldots k(k+1)})^2$ for $k \ge 2$,
and $\mu_{0h} = \max \{n!p_n, |(n-1)!p_{n-1}|, \ldots, |p_0|, n\}$. 
We know $n!p_n = m_0^nK_F^n > n$.
Note that the somewhat involved upper bound $\ell_0^*$ only depends on 
$h$ and is effective.

(3) 
For every integer $m \ge 2$, we set
$$
	\del(m) 
	= (n(2g-2+s) + s) \cdot m \cdot (m^nK_F^n + 1) \cdot h(m).
$$
We will mostly use $\del(m_0)$ and $\del(m_0\ell_0)$.
This is an essential term in our effective estimate and comes from
a theorem of Bedulev-Viehweg \cite[Theorem 1.4(c)]{BedVie}, which, at least for $m \ge m_0$, will yield
$$
	\deg f_*\CO_X(mK_{X/B}) \le \del(m).
$$
We recall that $2g-2+s > 0$ by \cite[Theorem 1.4(a)]{BedVie}.

(4)
For every integer $a \ge 2$ and $k = 1$ or $\ell_0$, we set
$$
	d(k, a) = \del(m_0k) + k (2g-2+a) h(m_0k).
$$
\end{notation}

\begin{definition}\label{def_const}
Based on the above, we now let $a=2$ and define the integers in 
Theorem \ref{effbir} explicitly as follows:
\begin{equation*} 
\begin{aligned}
N & = d(1,2) + (1-g)h(m_0) - 1,  \\
d & = d(1,2) (\ell_0^*+1)^{h(m_0)-n-1} 
		+ (h(m_0)-n-1)(d(\ell_0^*, 2) + 2g) (\ell_0^*+1)^{h(m_0)-n-2}.
\end{aligned}
\end{equation*} 
Note that these $N,d$ coincide with the constants $N,d$ defined in the Introduction in the statement of the main result.
\end{definition}

Now we prepare for the proof of Theorem \ref{effbir}.

\begin{setup} \label{setup1}
Let $A$ be an ample divisor on $B$ with $\deg A =a \ge 2$,
and let 
$$
	L = f^*(K_B + A) + m_0 K_{X/B}.
$$
Let $N_0 := h^0(X, \CO_X(L)) - 1$.
Let $E = f_*\CO_X(L)$ be a vector bundle of rank $r = h(m_0)$,
$\pi : \BP(E) \lra B$ the $\BP^{r-1}$-bundle associated to $E$, 
$\CO(1)$ the universal quotient line bundle for $\pi$, and 
$H$ a divisor on $\BP(E)$ with $\CO_{\BP(E)}(H) = \CO(1)$.
\end{setup}

Moreover, we use the following notations.
Let $\w_B = \CO_B(K_B)$, $\w_{X/B}^{m_0} = \CO_X(m_0K_{X/B})$, $\CA = \CO_B(A), \CL = \CO_X(L)$.
We denote, as usual, by $\Phi_{|L|} : X \dasharrow \BP^{N_0}$ 
the rational map associated to the complete linear system $|L|$, and by 
$\Phi_{|L|}(X)$ the closure $\ol{\Phi_{|L|}(X \sm \text{Bs}\,|L|)} \subset \BP^{N_0}$,
where $\text{Bs}\,|L|$ is the base locus of the linear system.  

The next proposition gives a more explicit form of Theorem \ref{effbir}. 
In the case $g \ge 2$, we can take $A = (m_0-1)K_B$ above, then 
$L = m_0K_X$ and $\Phi_{|L|}$ is the $m_0$-th pluricanonical map.
Hence if we put $a = (m_0-1)(2g-2)$ instead of $a = 2$ in Definition \ref{def_const},
we have the bounds with respect to $\Phi_{|m_0K_X|}$.
In any case, every smooth fiber $F$ is embedded by $|m_0K_F|$.

\begin{prop} \label{bir}
In Setup \ref{setup1}, one has:

(1) $h^0(X, \CL) = h^0(\BP(E), \CO(1))$, and
$N_0 \le d(1,a) + (1-g) h(m_0) -1$.  

(2)
$\Phi_{|L|} : X \dasharrow \BP^{N_0}$ gives an embedding 
on $X \setminus f^{-1}(S)$.

(3) 
$\Phi_{|H|} : \BP(E) \lra \BP^{N_0}$ gives an embedding. 

(4)
The natural homomorphism $\pi^*E \lra \CL$ is surjective on 
$X \sm f^{-1}(S)$, and the induced rational map 
$\vph_0 : X \dasharrow \BP(E)$ gives an  embedding on 
$X \setminus f^{-1}(S)$ with $\Phi_{|L|} = \Phi_{|H|} \circ \vph_0$.

(5) $\deg \Phi_{|L|}(X)$ is no greater than 
$$ d(1,a) (\ell_0^*+1)^{h(m_0)-n-1} 
		+ (h(m_0)-n-1)(d(\ell_0^*, a) + 2g) (\ell_0^*+1)^{h(m_0)-n-2}.
$$
\end{prop}

\begin{proof}
(0)
We first note that $E = f_*\CL = \w_B \ot \CA \ot f_*\w_{X/B}^{m_0}$ 
commutes with arbitrary base change on $B \setminus S$ 
(cf.\ \cite[Lemma 2.40]{Vbook}).
In our case, this is simply due to \cite[Theorem III.12.11]{Hartshorne} and 
$H^i(F, \CL|_F) \cong H^i(F, \w_{X/B}^{m_0}|_F) 
\cong H^i(F, \w_F^{m_0}) = 0$ for any $i > 0$ and any smooth fiber $F$.
In particular, the base change map:\ $f_*\CL \ot \CO_B/m_P^k 
\lra H^0(X_P, \CL \ot \CO_X/\CI_{X_P}^k)$ is an isomorphism
for any point $P \in B \setminus S$ and for any positive integer $k$, 
where $m_P$ (respectively $\CI_{X_P}$) is the ideal sheaf of 
$P$ in $B$ (respectively $X_P$ in $X$).

(1) 
It is immediate that $h^0(X, \CL) = h^0(B, E) = h^0(\BP(E), \CO(1))$.
We shall estimate $h^0(B, \w_B \ot \CA \ot f_*\w_{X/B}^{m_0}) = N_0-1$. The key ingredient is an estimate of $\deg f_*\w_{X/B}^{m_0}$ 
due to Bedulev-Viehweg. In fact, we can apply \cite[Theorem 1.4(c)]{BedVie} to obtain
$$
	\deg f_*\w_{X/B}^{m_0}
	\le (n(2g-2+s) + s) \cdot m_0 \cdot e(m_0) \cdot r(m_0).
$$
Here, $r(m_0) = \rank f_*\w_{X/B}^{m_0} = h(m_0)$, and 
$e(m_0) = e(m_0K_F)$ is a positive integer defined for
a very ample divisor $m_0K_F$ on a general fiber $F$.
A positive integer $e(R)$ in general is defined for 
an ample divisor $R$ on a smooth projective variety of dimension $n$, 
and $e(R)$ reflects the geometry of the linear system $|R|$.
Instead of recalling the definition of $e(R)$, we recall an estimate
in \cite[Corollary 5.11]{Vbook}:\ if $R$ is very ample, then $e(R) \le R^n +1$.
In our case, since $|m_0K_F|$ is very ample, we have
$e(m_0) \le m_0^nK_F^n + 1$.
Hence, we obtain
$$
	\deg f_*\w_{X/B}^{m_0}
	\le (n(2g-2+s) + s) \cdot m_0 \cdot (m_0^nK_F^n + 1) \cdot h(m_0)
	= \del(m_0).
$$
We can replace $m_0$ by any $m \ge m_0$ in the argument above,
and have $\deg f_*\w_{X/B}^m \le \del(m)$.

On the other hand, it is known (see, e.g., \cite[Proposition 1.3]{BedVie}) 
that $f_*\w_{X/B}^{m_0}$ is ample, because of the non-isotriviality of 
$f$ and the ampleness of $K_F$. (The weaker statement that ``$f_*\w_{X/B}^{m_0}$ is nef''
is enough if $\deg A \ge 3$, which is due to Kawamata \cite{Ka82}.) Thus, the vector bundle $\CA \ot f_*\w_{X/B}^{m_0}$ is also ample, and
in particular $H^1(B, \w_B \ot \CA \ot f_*\w_{X/B}^{m_0}) = 0$ (see Remark \ref{H1_van_remark}).
Then, by Riemann-Roch on $B$, we have 
\begin{equation*} 
\begin{aligned}
h^0(B, \w_B \ot \CA \ot f_*\w_{X/B}^{m_0})
& = \deg (\w_B \ot \CA \ot f_*\w_{X/B}^{m_0})
	 + (1-g)\, \rank (\w_B \ot \CA \ot f_*\w_{X/B}^{m_0}) \\
& = \deg f_*\w_{X/B}^{m_0} + (2g-2+a) h(m_0) + (1-g) h(m_0).
\end{aligned}
\end{equation*} 
Combining with the estimate for $\deg f_*\w_{X/B}^{m_0}$,
we have our estimate for $N_0$.

Using 
$\deg E = \deg f_*\w_{X/B}^{m_0} + \deg(\w_B \ot A) \, \rank f_*\w_{X/B}^{m_0}$
and the same reasoning as above, we have
$$
	\deg E \le \del(m_0) + (2g-2+a)h(m_0) = d(1,a).
$$

(2)
Let $P$ and $Q$ be two points on $B$, not necessarily distinct.
By the same token as above, we have
$H^1(B, \w_B \ot \CA \ot f_*\w_{X/B}^{m_0} \ot \CO_B(-P-Q)) = 0$.
Then the restriction map
\begin{equation*}\tag{$*$}
	H^0(X, \CL) \cong H^0(B, E) 
	\lra H^0(B, E \ot \CO_B/(m_P \cdot m_Q))
\end{equation*}
is surjective.
For the rest of this part (2), we assume $P, Q \in B \setminus S$.

(2.1) We consider the case $P \ne Q$.
Then by the base change property, 
$$
	H^0(B, E \ot \CO_B/(m_P \cdot m_Q))
	\cong H^0(X_P, \w_{X_P}^{m_0}) 
			\oplus H^0(X_Q, \w_{X_Q}^{m_0}). 
$$	
Since $|m_0K_{X_P}|$ and $|m_0K_{X_Q}|$ are very ample,
we can see, by varying $P$ and $Q$ in $B \setminus S$ with $P \ne Q$
in the surjection ($*$), that the map 
$\Phi_{|L|} : X \dasharrow \BP^{N_0}$ is regular on 
$X \setminus f^{-1}(S)$, and bijective on $X \setminus f^{-1}(S)$
onto its image.
Moreover, on every smooth fiber $F$, the restriction
$\Phi_{|L|}|_F : F \lra \BP^{N_0}$ gives an embedding 
by $|m_0 K_F|$.

(2.2) 
We would like to show that 
$\Phi_{|L|} : X \dasharrow \BP^{N_0}$ is an embedding on 
$X \setminus f^{-1}(S)$.
We take a point $x \in X \setminus f^{-1}(S)$, and shall show that
$H^0(X, \CL)$ generates tangent vectors at $x$.
We let $P = f(x)$. In (2.1) above, we showed that $H^0(X, \CL)$ generates 
tangent vectors at $x$ which are tangent to the fiber $X_P$.
So it is enough to find an element in $H^0(X, \CL)$ 
which generates a horizontal (with respect to $f : X \lra B$) 
tangent vector at $x$, i.e., a tangent vector which is not tangent to the fiber $X_P$.
To this aim, we consider the case $P = Q$ in the above, and we take 
an appropriate affine open subset $U \subset B \setminus S$ around $P$,
and a local coordinate $t_P$ on $U$ centered at $P$.
We can regard $t := f^*t_P$ as part of a local coordinate system of $X$
centered at $x$. We observe that $\frac{\rd}{\rd t}$ is a global generator of 
the normal bundle $N_{X_P/X}$ of $X_P$, 
and denote by $[t]$ the image of $t$ in $\CI_{X_P}/\CI_{X_P}^2$.
Then, by the base change property,
$$
	H^0(B, E \ot \CO_B/(m_P \cdot m_Q))
	\cong H^0(X_P, \w_{X_P}^{m_0}) 
			\oplus H^0(X_P, \w_{X_P}^{m_0} \ot \CI_{X_P}/\CI_{X_P}^2).
$$	
We take $\sigma_P \in  H^0(X_P, \w_{X_P}^{m_0})$ with $\sigma_P(x) \ne 0$. 
We take an extension $\sigma_U \in H^0(X_U, \CL)$ of $\sigma_P$,
where $X_U = f^{-1}(U)$.
This is possible due to the base change property (0).
We consider $\sigma_U \cdot t \in H^0(X_U, \CL)$,
which defines by restriction a non-zero element of 
$H^0(B, E \ot m_P/m_P^2)$.
By the surjection ($*$), we have an extension 
$\widetilde\sigma \in H^0(X, \CL)$ of $\sigma_U \cdot t$.
Since $\widetilde\sigma|_{X_U} - \sigma_U \cdot t
\in H^0(X_U, \CL \ot \CI_{X_P}^2)$, we have
$(\frac{\rd}{\rd t} \widetilde\sigma)|_{X_P} 
= (\frac{\rd}{\rd t} (\sigma_U \cdot t))|_{X_P} 
= \sigma_U|_{X_P}$.
Thus we have $(\frac{\rd}{\rd t} \widetilde\sigma)(x) \ne 0$.

(3)
Recall $r = \rank E = h(m_0)$. Clearly, $r > 1$. We note the base change property for $E = \pi_*\CO(1)$, due to the
fact that $H^1(\pi^{-1}(P), \CO(1)) = H^1(\BP^{r-1}, \CO_{\BP^{r-1}}(1)) = 0$
for any $P \in B$. Again, recall that $H^1(B, E \ot \CO_B(-P-Q)) = 0$ for any 
$P, Q \in B$, not necessarily distinct. Hence the restriction map
\begin{equation*}\tag{$*'$}
	H^0(\BP(E), \CO(1)) \cong H^0(B, E) 
	\lra H^0(B, E \ot \CO_B/(m_P \cdot m_Q))
\end{equation*}
is surjective for any $P, Q \in B$.
On every $\pi^{-1}(P)$, we have of course
$H^0(\pi^{-1}(P), \CO(1)) = H^0(\BP^{r-1}, \CO_{\BP^{r-1}}(1))$,
and see that $|H_{|{\pi^{-1}(P)}}|$ is very ample.
The remaining arguments to obtain the very ampleness of $|H|$ are 
the same as in (2) above.

(4)
On $X \sm f^{-1}(S)$, we have $\Phi_{|L|} = \Phi_{|H|} \circ \vph_0$,
because of $(\Phi_{|H|} \circ \vph_0)^* \CO_{\BP^{N_0}}(1)
= \vph_0^* (\Phi_{|H|}^* \CO_{\BP^{N_0}}(1))
= \vph_0^*\CO(1) = \CL$ over $X \sm f^{-1}(S)$.
Since $\Phi_{|L|}$ gives an embedding on $X \sm f^{-1}(S)$,
so does $\vph_0$.

(5)
This degree bound will be given separately in Lemma \ref{degbd}, where we may clearly replace $\ell_0$ by its upper bound $\ell_0^*$.
\end{proof}
\begin{rem}\label{H1_van_remark}
In the proof of Proposition \ref{bir}(1), the following vanishing of cohomology was used: Let $E$ be an ample vector bundle on $B$. 
Then $H^1(B, \w_B \ot E) = 0$. To give a proof by contradiction, assume that 
$H^1(B, \w_B \ot E) \not = 0$. By Serre duality, this implies 
$H^0(B, E^*) \ne 0$ and thus $H^0(B, S^k(E^*)) \ne 0$ 
for any positive $k$, where $S^k(E^*)$ is the $k$-th symmetric tensor. Applying Serre duality again, we obtain $$0\not = H^0(B, S^k(E^*))=H^1(B, \w_B \ot S^k(E)).$$
However, this is a contradiction to the cohomological characterization of ample vector bundles (\cite[Theorem 6.1.10]{PAGII}).
\end{rem}

We devote the rest of this section to proving the effective degree bound
of $\Phi_{|L|}(X) \subset \BP^{N_0}$, stated  in\ Proposition \ref{bir}(5).
We first fix some notations and make remarks.

\begin{rem} \label{remH}
(1)
We let $X' := \vph_0(X) \subset \BP(E)$ with reduced structure,
and let $f' : X' \lra B$ be the induced morphism.
We denote by $\CI_{X'} \subset \CO_{\BP(E)}$ the ideal sheaf of $X'$,
and let $\CI_{X'}(k) = \CI_{X'} \ot \CO_{\BP(E)}(k)$ for every integer $k$.

(2)
Since $H$ is very ample on $\BP(E)$ and 
$\Phi_{|L|} = \Phi_{|H|} \circ \vph_0$, we have
$\deg \Phi_{|L|}(X) = X' \cdot H^{n+1}$.
Thus we shall estimate the intersection number $X' \cdot H^{n+1}$.

(3)
In the course of the proof of Proposition \ref{bir},
we observed that
$\deg f_*\w_{X/B}^m \le \del(m)$ for any $m \ge m_0$, and also that, with $r = \dim \BP(E)  = \rank E = h(m_0)$, 
the top self-intersection number $H^r = \deg E \le d(1,a)$.
\end{rem}

To bound the degree $X' \cdot H^{n+1}$, we aim to find hypersurfaces
in $\BP(E)$ with ``degree bound.'' The precise statement is

\begin{lem}
\label{glogen}
Let $P_0 \in B$ be a point.
Then $\CI_{X'}(\ell_0) \ot \pi^*\CO_B((d(\ell_0,a)+2g)P_0)$ 
is generated by global sections.
\end{lem}

Taking Lemma \ref{glogen} for granted for a moment, 
we state the final estimate.

\begin{lem} \label{degbd}
The degree is bounded by 
$$
	\deg \Phi_{|L|}(X) = X' \cdot H^{n+1} 
	\le \big(\ell_0+1\big)^{r-n-1} H^r 
	     + (r-n-1) \big(d(\ell_0,a)+2g\big) \big(\ell_0+1\big)^{r-n-2}
$$
with $r = \rank E = h(m_0)$ and $H^r = \deg E \le d(1,a)$.
\end{lem}

\begin{proof}
We let $d_0 = d(\ell_0,a)$, and $c = r-n-1$
the codimension of $X'$ in $\BP(E)$.
We claim that
$$
	\big((\ell_0+1)H + (d_0+2g) \pi^*P_0 \big)^c \equiv X' +Z'
$$
for some effective $\BQ$-coefficient $(n+1)$-dimensional cycle $Z'$
on $\BP(E)$, where $\equiv$ stands for numerical equivalence.
Taking this for granted for a moment, we can see
$X' \cdot H^{n+1} \le (X' + Z') \cdot H^{n+1}
= ((\ell_0+1)H + (d_0+2g) \pi^*P_0)^{r-n-1} \cdot H^{n+1}
= (\ell_0+1)^{r-n-1} H^r + (r-n-1) ((\ell_0+1)H)^{r-n-2} 
	\cdot (d_0+2g) \pi^*P_0 \cdot H^{n+1}
= (\ell_0+1)^{r-n-1} H^r + (r-n-1) (d_0+2g) (\ell_0+1)^{r-n-2}$.

Let us prove the claim.
The argument here is inspired by that of \cite[Lemma 7.2]{Ha75}.
We take a log-resolution $\mu : Y \lra \BP(E)$ of the ideal sheaf $\CI_{X'}$
by successive blowing-ups along non-singular centers. The domain $Y$ is 
a smooth projective variety, and $\mu$ is isomorphic on $\BP(E) \sm X'$. 
Moreover, $\mu^{-1}\CI_{X'} = \CO_Y(-D)$ for an effective divisor $D$
with simple normal crossing support.
We denote by $D = \sum_{i \in I} a_iD_i$ the decomposition 
into irreducible components with positive integer coefficients $a_i$. By the global 
generation of 
$\CI_{X'}(\ell_0) \ot \pi^*\CO_B((d_0+2g)P_0)$ established in Lemma \ref{glogen},
the linear system $|\mu^*(\ell_0 H + (d_0+2g)\pi^*P_0) - D|$ 
is base point free.
On the other hand, since $H$ is ample, there exist non-negative rational
numbers $b_i$ such that the $\BQ$-divisor $\mu^*H - \sum_{i \in I} b_iD_i$ 
is ample.
Hence, the $\BQ$-divisor 
$$
	G := \mu^*\big(\ell_0 H + (d_0+2g)\pi^*P_0 \big) - D 
			+ \mu^*H - \sumn_{i \in I} b_iD_i
$$
is ample, being the sum of a semi-ample divisor (whose corresponding linear
system is in fact base point free) and an ample $\BQ$-divisor.
We take a large and sufficiently divisible integer $k$ such that
all $kb_i$ become integers and $kG$ is very ample. We then take general members
$B_1, \ldots, B_c \in |kG|$ so that
$B_1 \cap \ldots \cap B_c$ is a smooth irreducible $(n+1)$-dimensional
variety.
Then $B_j + k(D + \sum_{i \in I} b_iD_i) \in 
|k\mu^*((\ell_0 +1) H + (d_0+2g)\pi^*P_0)|$ for every $j$.
Thus there exists
$A_j \in |k((\ell_0 +1) H + (d_0+2g)\pi^*P_0)|$ such that
$\mu^*A_j = B_j + k(D + \sum_{i \in I} b_iD_i)$ for every $j$.
This in particular implies that the order of vanishing of every 
$A_j$ along $X'$ is at least $k$.
Thus $k^c ((\ell_0 +1) H + (d_0+2g)\pi^*P_0)^c \equiv k^cX' + Z'_k$
for an effective $(n+1)$-dimensional cycle $Z'_k$ (whose support is contained
in $X' \cup \mu(B_1 \cap \ldots \cap B_c)$) on $\BP(E)$.
By dividing by $k^c$, we have our claim.
\end{proof}

Let us discuss the global generation Lemma \ref{glogen}.
We note that $\vph_0 : X \dasharrow X'$ is biregular over $B \sm S$,
and that $X'$ may be singular along ${f'}^{-1}(S)$.
On the other hand, $\CO_{X'}(1) := \CO(1)|_{X'}$ is very ample,
and $f' : X' \lra B$ is a flat family of subschema of $\BP^{r-1}$ with
Hilbert polynomial $\chi(X'_P, \CO_{X'_P}(m))$ (\cite[Proposition III.9.7, Theorem III.9.9]{Hartshorne}), 
where $X'_P = {f'}^*P$ is the scheme theoretic fiber for $P \in B$.
Since $\CO_{X'_P}(1) \cong \w_{X_P}^{m_0}$ if $P \in B \sm S$,
the Hilbert polynomial $\chi(X'_P, \CO_{X'_P}(m))$ is $h_0(m) = h(m_0m)$.
For the original $f : X \lra B$, although {\it smooth} fibers
have the same Hilbert polynomial $h(m)$, we did not have a natural
way to make {\it all} fibers have the same Hilbert polynomial.

The next lemma, essentially due to Gotzmann, on Castelnuovo-Mumford
regularity will give a surprising input in our effective estimate.
Recall that $\ell_0$ is the length of the binomial sum expansion of 
the Hilbert polynomial $h_0(m) = h(m_0m)$.

\begin{lem} \label{reg} 
For every scheme theoretic fiber $X'_P = {f'}^*P$ over $P \in B$,
the ideal sheaf $\CI_{X_P'} \subset \CO_{\BP^{r-1}}$ is $\ell_0$-regular.
In particular, $\CI_{X_P'}(\ell_0)$ is generated by global sections in $H^0(\BP^{r-1}, \CI_{X_P'}(\ell_0))$, 
$\pi_*\CI_{X'}(\ell_0)$ commutes with arbitrary base change, 
$R^1\pi_*(\CI_{X'}(\ell_0)) = 0$, and the natural sequence
$0 \lra \pi_*\CI_{X'}(\ell_0) \lra \pi_*\CO(\ell_0) 
\lra f'_*\CO_{X'}(\ell_0) \lra 0$
is exact.
\end{lem}

\begin{proof}
Every fiber of $f' : X' \lra B$ has the same Hilbert polynomial $h_0(m)$.
By a theorem of Gotzmann \cite{G} (\cite[Theorem 1.8.35]{PAGI}, \cite[Theorem 4.3.2]{BH}), 
every $\CI_{X_P'}$is $\ell_0$-regular.
By definition, $\CI_{X_P'}$is $\ell_0$-regular if
$H^i(\BP^{r-1}, \CI_{X_P'}(\ell_0-i)) = 0$ for all $i > 0$ (\cite[Definition 1.8.1]{PAGI}).
As a consequence, for every $k \ge \ell_0$, $\CI_{X_P'}(k)$ is generated by 
global sections, and $\CI_{X_P'}$ is $k$-regular (\cite[Theorem 1.8.3]{PAGI}).
From this, we obtain that, for any $P \in B$, $\CI_{X_P'}(\ell_0)$ is generated by 
global sections, and $H^1(\BP^{r-1}, \CI_{X_P'}(\ell_0)) = 0$
by the $(\ell_0+1)$-regularity.
In particular, the direct image sheaf $\pi_*\CI_{X_P'}(\ell_0)$
commutes with arbitrary base change, and hence every fiber at $P \in B$
is naturally isomorphic to $H^0(\BP^{r-1}, \CI_{X_P'}(\ell_0))$.
The vanishing $R^1\pi_*(\CI_{X'}(\ell_0)) = 0$ is a consequence of 
$H^1(\BP^{r-1}, \CI_{X_P'}(\ell_0)) = 0$ for any $P \in B$.
\end{proof}

\begin{lem} \label{incl} 
For every $k \ge 1$, there exists a natural injective homomorphism
$f'_*\CO_{X'}(k)$ $\lra f_*\CL^{\ot k}$, which is isomorphic on $B \sm S$.
\end{lem}

\begin{proof}
We take a birational morphism $\mu : \tX \lra X$ to resolve the singularities
of the rational map $\vph_0 : X \dasharrow X'$, and denote by 
$\mu' : \tX \lra X'$ the induced morphism.
We can take $\mu : \tX \lra X$ so that $\mu$ is biregular on 
$X \sm f^{-1}(S)$, and such that the image of the natural homomorphism
$(f \circ \mu)^*(f \circ \mu)_* (\mu^*\CL) \lra \mu^*\CL$ is 
$\mu^*\CL \ot \CO_\tX(-\Delta)$ for some effective divisor $\Delta$ on $\tX$
and $\mu^*\CL \ot \CO_\tX(-\Delta)$ is $\widetilde f$-generated 
for $\widetilde f := f \circ \mu = f' \circ \mu' : \tX \lra B$, 
i.e., the natural homomorphism 
$\widetilde f^*\widetilde f_* (\mu^*\CL \ot \CO_\tX(-\Delta)) 
\lra \mu^*\CL \ot \CO_\tX(-\Delta)$ is surjective. 
Since the induced composition
$(f \circ \mu)_* (\mu^*\CL) 
\lra (f \circ \mu)_*(\mu^*\CL \ot \CO_\tX(-\Delta))
\subset (f \circ \mu)_* \mu^*\CL$
is identical, we have 
$(f \circ \mu)_*(\mu^*\CL \ot \CO_\tX(-\Delta)) = (f \circ \mu)_* \mu^*\CL= E$.
Thus, this $\widetilde f$-generated line bundle $\mu^*\CL \ot \CO_\tX(-\Delta)$
defines a morphism $\tX \lra \BP(E)$ over $B$, which is nothing but
$\mu' : \tX \lra X' \subset \BP(E)$, and thus
${\mu'}^* \CO_{X'}(1) = \mu^*\CL \ot \CO_\tX(-\Delta)$.
Then for every $k \ge 1$, we have an injective sheaf homomorphism
${\mu'}^* \CO_{X'}(k) \lra \mu^*\CL^{\ot k}$ and 
$f'_*\CO_{X'}(k) \lra f_*\CL^{\ot k}$.
Since the support of $\Delta$ is contained in $\widetilde f^{-1}(S)$,
we have ${\mu'}^* \CO_{X'}(1) = \mu^*\CL$ over $B \sm S$, and hence
$f_*\CL^{\ot k} = f'_*\CO_{X'}(k)$ over $B \sm S$.
\end{proof}
We are now ready to prove Lemma \ref{glogen}.
\begin{proof}[Proof of Lemma \ref{glogen}]
We let $d_0 = d(\ell_0,a)$.

(1) 
We first establish ``how negative'' $\pi_*\CI_{X'}(\ell_0)$ is. Let $\pi_*\CI_{X'}(\ell_0) \lra \CM$ be a quotient line bundle with 
kernel $\CN$.
We claim $\deg \CM > -d_0$, i.e., there exists a uniform effective bound.

Since $\CN$ can be seen as a subbundle of $\pi_*\CO(\ell_0) = S^{\ell_0}(E)$
and $S^{\ell_0}(E)$ is ample, we have $\deg \CN < \deg \pi_*\CO(\ell_0)$.
Then $\deg \CM = \deg \pi_*\CI_{X'}(\ell_0) - \deg \CN
= \deg \pi_*\CO(\ell_0) - \deg f'_*\CO_{X'}(\ell_0) - \deg \CN 
> - \deg f'_*\CO_{X'}(\ell_0)
\ge - \deg f_*\CL^{\ot \ell_0}$.
For the second equality, we used the exact sequence in Lemma \ref{reg},
and for the last inequality we used Lemma \ref{incl}. 
Thus, it is enough to show that $\deg f_*\CL^{\ot \ell_0} \le d_0$.

Since $f_*\CL^{\ot \ell_0} 
= (\w_B \ot \CA)^{\ot \ell_0} \ot f_*\w_{X/B}^{m_0\ell_0}$, we have
$\deg f_*\CL^{\ot \ell_0} = \deg f_*\w_{X/B}^{m_0\ell_0} 
+ \ell_0(2g-2+a) \, \rank f_*\w_{X/B}^{m_0\ell_0}$.
The key is again \cite[Theorem 1.4(c)]{BedVie}, and we have
$\deg f_*\w_{X/B}^{m_0\ell_0} \le \del(m_0\ell_0)$ by Remark \ref{remH}.
Since $\rank f_*\w_{X/B}^{m_0\ell_0} = h(m_0\ell_0)$, we have
$\deg f_*\CL^{\ot \ell_0}\le d_0$.

(2)
Now, in view of (1), 
$\w_B^{-1} \ot \pi_*\CI_{X'}(\ell_0) \ot \CO_B((d_0+2g)P_0 - P - Q)$ is ample 
for any $P, Q \in B$ by Hartshorne's theorem \cite{Ha71} (\cite[Theorem 6.4.15]{PAGII}), 
because any quotient line bundle has positive degree. 
Thus, we have a vanishing
$H^1(B, \pi_*\CI_{X'}(\ell_0) \ot \pi^*\CO_B((d_0+2g)P_0 - P - Q)) = 0$ 
for any $P, Q \in B$.
Hence the restriction map
$$
	H^0(\BP(E), \CI_{X'}(\ell_0) \ot \CO_B((d_0+2g)P_0))
	\lra H^0(\BP^{r-1}, \CI_{X'_P}(\ell_0)) 
			\oplus H^0(\BP^{r-1}, \CI_{X'_Q}(\ell_0))
$$
is surjective, where $P \ne Q$ in this expression.
Here we used Lemma \ref{reg} that $\pi_*\CI_{X'}(\ell_0)$ commutes 
with arbitrary base change.
Since $\CI_{X'_P}(\ell_0)$ and $\CI_{X'_Q}(\ell_0)$ are generated by 
global sections by Lemma \ref{reg}, we also have the global generation of 
$\CI_{X'}(\ell_0) \ot \pi^*\CO_B((d_0+2g)P_0)$ on $\BP(E)$.
\end{proof}

\section{Proof of Theorem \ref{mt}}\label{shafsection}

As we mentioned in the Introduction, the bound $C(g,s,n,v)$ is easily derived from the bound $C(g,s,h)$ and Proposition \ref{coeff_bound_1}. For this reason, it suffices to work with $C(g,s,h)$ in this section.\par

Let $(B, S)$ and $h$ be as in Theorem \ref{mt}.
We first construct a projective variety $W$ determined from $(B, S)$ and $h$. To this end, we take an effective divisor $L_B$ on $B$ with 
$$
	\deg L_B = 2g+1 =: d_B.
$$
It is known that $L_B$ is very ample (\cite[Corollary IV.3.2]{Hartshorne}). 
By the Riemann-Roch theorem, $h^0(B,\OO_B(L_B))= g+2$.
Let
$$
	\varphi_2 = \Phi_{|L_B|} : B \lra \PP^{g+1}_B
$$
be the embedding by the complete linear system $|L_B|$.
To avoid ambiguities, we write $\PP^{g+1}_B$ for the codomain of $\varphi_2$.
Let $N = N(g,s,h)$ be the integer in Definition \ref{def_const}, and let
$$
	\sss : \PP^N \times \PP^{g+1}_B \lra \PP^M 
	\ \text{ with } \ M = M(g,s,h) = (N+1)(g+2)-1
$$
be the Segre embedding. 
We write down the Segre embedding in homogeneous coordinates as follows:
$$
	([X_0,\ldots,X_N],[Y_0,\ldots,Y_{g+1}])
	\mapsto [X_0Y_0,\ldots,X_0Y_{g+1},\ldots,X_NY_0,\ldots, X_NY_{g+1}].
$$
We write the homogeneous coordinates $[\ldots, X_{i,j}, \ldots]$ 
of $\PP^M$ so that the map $\sss$ is given by 
$X_{i,j} = X_iY_j$ for $0 \le i \le N$ and $0 \le j \le g+1$.
We identify $\BP^{g+1}_B$ (with homogeneous coordinates $[Y_0,\ldots,Y_{g+1}]$) and 
the linear subspace $\BP^{g+1} 
= \{X_{i,j} = 0$ for $1 \le i \le N$ and $0 \le j \le g+1 \} \subset \BP^M$
(with coordinates $[X_{0,0},\ldots,X_{0,g+1}]$).

Let $V = \{X_{0,j} = 0$ for $0 \le j \le g+1 \} \subset \BP^M$ be 
a linear subspace, and let 
$$
	\pi_V : \PP^M \dasharrow \PP^{g+1}_B
$$ 
be the projection from $V$ onto the first $g+2$ coordinates.
Let
$$
	W \subset \PP^M
$$
be the variety consisting of the union of lines joining $\varphi_2(B) (\subset 
\PP^{g+1}_B \subset \BP^M)$ and $V$.
It can also be written as 
$W = \overline{(\pi_V|_{\BP^M \setminus V})^{-1}(\varphi_2(B))}$
in $\BP^M$, where $\pi_V|_{\BP^M \setminus V} : \BP^M \setminus V
\lra \BP^{g+1}_B$ is holomorphic.

\begin{lem}\label{W}
The subvariety $W$ is defined by equations of degree no more than $d_B=2g+1$.
\end{lem}

\begin{proof}
The degree of $\varphi_2(B)$ in $\PP^{g+1}_B$ is equal to $d_B$. 
It is well-known (see \cite[Proposition 1.14(a), Remark 1.17]{Catanese_JAG}) that there is a finite set of homogeneous polynomials of degree $d_B$, 
denoted $\{\tau_\ga(Y_0,\ldots, Y_{g+1}) 
\in H^0(\BP^{g+1}_B, \CO_{\BP^{g+1}_B}(d_B))\}_{\alpha}$, such that
$$
	\varphi_2(B)=\bigcap_{\alpha}\{\tau_\ga=0\},$$
both set-theoretically and scheme-theoretically. We now lift these $\tau_\ga$ to 
$\widetilde \tau_\ga \in H^0(\BP^M, \CO_{\BP^M}(d_B))$
by letting
$$
\widetilde \tau_\ga(X_{0,0},\ldots,X_{0,g+1},\ldots,X_{N,0},\ldots,
   X_{N, g+1}) = \tau_\ga(X_{0,0},\ldots, X_{0,{g+1}}).
$$
Then the sections 
$\{\widetilde \tau_\ga \in H^0(\BP^M, \CO_{\BP^M}(d_B))\}_\alpha$ 
define the subvariety $W \subset \PP^M$.
\end{proof}
We quote a result due to Guerra, which can be formulated in 
a slightly more general setting as follows. Let, in general, ${\Chow}_{\kappa,\delta}(W)$ be the Chow variety of 
$\kappa$-dimensional subvarieties of degree $\delta$ which are contained
in $W \subset \BP^M$. 
Let ${\Chow}_{\kappa,\delta}'(W)$ denote the union of those irreducible 
components of ${\Chow}_{\kappa,\delta}(W)$ whose general points represent 
irreducible cycles.  Then the following general Proposition is proven in \cite[Proposition 2.4]{Guerra}
based on an argument from \cite[Exercise I.3.28]{Kollarbook}.
\begin{proposition}\label{Chowbound}
Let $\kappa, \delta_1$ and $\delta_2$ be positive integers.
Let in general $W\subset \P^M$ be a projective variety defined by equations 
of degree no more than $\delta_1$. 
Then the number of irreducible components of $\Chow'_{\kappa,\delta_2}(W)$ 
is no more than
\begin{equation*}
{(M+1)\max\{\delta_1,\delta_2\}\choose M}^{(M+1)\left(\delta_2 {\delta_2+\kappa-1\choose \kappa }+{\delta_2+\kappa-1 \choose \kappa-1}\right)}.
\end{equation*} 
\end{proposition}

We next consider an admissible family $f : X \lra B$ for $(B,S)$ and $h$.
We modify the rational map $\Phi_{|L|} : X \dasharrow \BP^N$ (with $a=2$) obtained in 
Proposition \ref{bir} (perhaps after a linear inclusion $\BP^{N_0} \subset 
\BP^N$) to a form which respects the fibration $f : X \lra B$.
Let
$\Phi_{|L|} \times (\varphi_2 \circ f) : X \dasharrow \BP^N \times \BP^{g+1}_B$
be the induced map. Note that it is immediate that $\varphi_2 \circ f =  \Phi_{|f^*L_B|}$.
We then compose with the Segre embedding $\sss : \PP^N \times \PP^{g+1}_B \lra \PP^M$.
More concretely, we take a basis $\sg_0,\ldots,\sg_{N_0}$ for 
$H^0(X,\OO_X(L))$, $\sg_i = 0$ for ${N_0} < i \le N$,
and a basis $s_0,\ldots,s_{g+1}$ for $H^0(B,\OO_B(L_B))$, and let 
$\varphi_1':= \sss \circ (\Phi_{|L|} \times \Phi_{|f^*L_B|}) 
: X \dasharrow \BP^M$
by
$$
	x \mapsto [\sg_0f^*s_0,\ldots,\sg_0f^*s_{g+1},\ldots,
				\sg_Nf^*s_0,\ldots,\sg_Nf^*s_{g+1}](x).
$$

\begin{lemma}\label{deg_bound}
The map
$$
	\varphi_1 := \sss \circ (\Phi_{|L|} \times \Phi_{|f^*L_B|})
	: X\setminus f^{-1}(S) \lra \PP^M
$$
is an embedding of $X\setminus f^{-1}(S)$. Moreover, if we denote by
$$
	Z_f = \ol{\varphi_1(X \setminus f^{-1}(S))} \subset \PP^M
$$
the Zariski closure of $\varphi_1(X\setminus f^{-1}(S))$.
Then 
$$
	\deg Z_f = (n+1)(2g+1) \deg \Phi_{|L|}(X).
$$
\end{lemma}

\begin{proof}
It is clear that $\varphi_1$ is an embedding of $X\setminus f^{-1}(S)$ due 
to the fact that the first component map $\Phi_{|L|}$ already is an embedding 
by itself.\par
Let $\pr_1:\PP^N \times \PP^{g+1}_B\lra \PP^N$ and $\pr_2:\PP^N \times \PP^{g+1}_B\lra \PP^{g+1}_B$ be the first and second projections. Let $\CH_1$ resp. $\CH_2$ be the hyperplane line bundles in $\PP^N$ resp. $ \PP^{g+1}_B$. Then
\begin{align*}
\deg Z_f
&= \big((\pr_1^*{\CH_1}+ \pr_2^*\CH_2)
		|_{\ol{(\Phi_{|L|} \times \Phi_{|f^*L_B|})(X\setminus f^{-1}(S))}} \big)^{n+1}\\
&= (n+1)(\CH_1|_{{\Phi_{|L|}(X)}})^n 
	\cdot  \deg \CH_2 |_{\Phi_{|L_B|}(B)}\\
&= (n+1)\big(\deg \Phi_{|L|}(X)\big)(2g+1).
\end{align*}
\end{proof}
By the construction of $\varphi_1$, there is a commutative diagram 
$$
\begin{CD}
X\setminus (\{\sigma_0 = 0\}\cup f^{-1}(S)) @>\iota>>X\setminus f^{-1}(S) @>\varphi_1 >> \PP^M\\
@ V f VV @V f VV @V{\pi_V}  VV\\
B\setminus S@>=>>  B\setminus S@>\varphi_2|_{B\setminus S}>> \PP^{g+1}_B\\
\end{CD}.
$$
Here, $\iota$ denotes the inclusion map and the vertical map
$\pi_V$ on the right hand side is merely rational.
Since the rational map $\pi_V \circ \varphi_1 : X \setminus f^{-1}(S) 
\dasharrow \BP^{g+1}_B$ is given by $x\mapsto [\sg_0f^*s_0, \ldots, \sg_0f^*s_{g+1}](x)$,
the restriction map 
$$ 
  \pi_V: \varphi_1(X\setminus (\{\sigma_0= 0\}\cup f^{-1}(S)))
  		\lra \varphi_2(B\setminus S)
$$
is holomorphic by construction.
Moreover, from the expression $[\sg_0f^*s_0, \ldots, \sg_0f^*s_{g+1}]$,
we see that the singularity of the rational map $\pi_V \circ \varphi_1$ 
along the divisor $\{\sg_0 = 0\}$ in $X \setminus f^{-1}(S)$ is removable.
It is extended by letting $[\sg_0f^*s_0, \ldots, \sg_0f^*s_{g+1}](x) 
= [f^*s_0, \ldots, f^*s_{g+1}](x) = [s_0, \ldots, s_{g+1}](f(x))$
in $\BP^{g+1}_B$ for $x \in X \setminus f^{-1}(S)$.
This in particular implies $\pi_V \circ \varphi_1 = \varphi_2 \circ f$ 
holds on $X \setminus f^{-1}(S)$. 
Thus we have

\begin{lemma}\label{cd_lemma}
The holomorphic map $\pi_V: \varphi_1(X\setminus (\{\sigma_0= 0\}\cup f^{-1}(S)))\lra \varphi_2(B\setminus S)$ can be extended to a holomorphic map
$$
	\pi_V:\varphi_1(X\setminus f^{-1}(S))\lra \varphi_2(B\setminus S)
$$
such that the diagram
$$
\begin{CD}
X\setminus f^{-1}(S) @>\varphi_1 >> \varphi_1(X\setminus f^{-1}(S))\\
@V f VV @V{\pi_V}  VV\\
B\setminus S@>\varphi_2>> \varphi_2({B\setminus S})\\
\end{CD}
$$
is a commutative diagram of holomorphic maps. In fact, the diagram is an isomorphism of families over $B \setminus S$.
\end{lemma}

The holomorphic map $\pi_V:\varphi_1(X\setminus f^{-1}(S))\lra \varphi_2(B\setminus S)$
can be seen as an embedded projective model for 
$f : X \setminus f^{-1}(S) \lra B \setminus S$ with effective degree bounds. We shall now bound the possible deformations types of this family. Since 
$$
	Z_f = \overline{\varphi_1(X\setminus f^{-1}(S))}\subset W \subset \BP^M,
$$
$Z_f$ corresponds to a point in ${\Chow}_{n+1,\delta}'(W)$ with 
$$
	\delta := (n+1)(2g+1)\deg \Phi_{|L|}(X)
$$ 
due to Lemma \ref{deg_bound}.
When we apply Proposition \ref{Chowbound} to our situation, we find that 
$d_B = 2g+1 = \delta_1 < \delta_2 = \delta$ due to Lemma \ref{W}.
Therefore, the number of irreducible components of 
${\Chow}_{n+1,\delta}'(W)$ is no more than
\begin{equation*}
{(M+1)\delta\choose M}
^{(M+1)\left(\delta {\delta+n\choose n+1 }+{\delta+n \choose n}\right)}.
\end{equation*} 
Our main Theorem \ref{mt} for the case of the bound $C(g,s,h)$ now 
follows from the following Proposition.
Recall that $\deg \Phi_{|L|}(X) \le d$,
where $d = d(g,s,h)$ is the integer in Definition \ref{def_const}.

\begin{proposition}\label{final_prop}
The total number of irreducible components of the Chow varieties 
$$
	{\Chow}_{n+1,(n+1)(2g+1)\nu}'(W),\quad  \nu=1,\ldots,d,
$$
is an upper bound for the number $C(g,s,h)$ of deformation types 
in Theorem \ref{mt}.
\end{proposition}

The proof of Proposition \ref{final_prop} is identical to the proof of the corresponding \cite[Proposition 2.11]{Heier_Crelle}, so we do not repeat is here.

\section{Effective bounds on Hilbert polynomials}

In this final section, we shall give the outstanding proofs of some effective bounds regarding
Hilbert polynomials, which were used in the proof of our main result.
\subsection{The bound on length}
We give an effective bound for $\ell_0$, i.e.,\ the length of the binomial sum expansion as defined in Notation \ref{numbers}(2), in a general context.\par
Let $F \subset \BP$ be a closed subscheme of dimension $n$
in a projective space $\BP$ with the ample generator $\CO(1)$
of the Picard group.
Let $P(x) \in \BQ[x]$ be the Hilbert polynomial of $F$ with respect to 
$\CO(1)$, i.e., $P(m) = \chi(F, \CO_F(m))$ holds for all 
sufficiently large integer $m$. By a theorem of Gotzmann \cite{G} (\cite[Theorem 1.8.35]{PAGI}, \cite[Theorem 4.3.2]{BH}), 
there exists a unique sequence of integers 
$a_1 \ge a_2 \ge \ldots \ge a_\ell \ge 0$ such that
$$
   P(x) = \binom{x+a_1}{a_1} + \binom{x+a_2-1}{a_2} + \ldots 
			+ \binom{x+a_\ell-(\ell-1)}{a_\ell}.
$$
We write $P(x) = p_nx^n + p_{n-1}x^{n-1} + \ldots + p_1x + p_0$
with $p_i \in \BQ$.
Noting $\binom{x+a-j}{a} = x^a/a! + \text{(lower terms)}$, we see that
the sequence starts with $a_j = n$ for $1 \le j \le n!p_n$, and 
$a_j < n$ for $j > n!p_n$.
In view of this, we set $\ell_{n+1} = 0$, and
$$
	\ell_k = \max \{j \ge 0; \ a_j \ge k \}
$$
for $k = n, n-1, \ldots, 0$.
Then $0 = \ell_{n+1} < \ell_n = n!p_n \le \ell_{n-1} \le \ldots 
\le \ell_1 \le \ell_0$,
and $\ell_0$ is the length of $P(x)$.
The bound of $\ell_0$ in Notation \ref{numbers}(2) is a consequence
of the following

\begin{lem} \label{ell}
One can compute $\ell_n, \ell_{n-1}, \ldots, \ell_0$ recursively in terms of 
$p_n, p_{n-1}, \ldots, p_0$ and $n$.
If one prefers an explicit effective bound, one has for example
$$
	\ell_0 \le \sumn_{k=0}^n \gc_k \mu_P^{(k+1)!},
$$
where $\gc_0 = 1, \gc_1=2$, 
$\gc_k = k^{k+1}\gc_{k-1}^{k+1} 
= k^{k+1} (k-1)^{k(k+1)} \ldots 3^{4 \cdot 5 \ldots k(k+1)}
(2^{3 \cdot 4 \ldots k(k+1)})^2$ for $k \ge 2$ 
(the last factor is exceptional),
and $\mu_P = \max \{n!p_n, |(n-1)!p_{n-1}|, \ldots, |p_0|, n\}$.
\end{lem}

\begin{proof}

(1) 
Let $P(x) = \sum_{k=0}^n Q_k(x)$ with 
$Q_k(x) = \sum_{\ell_{k+1} < j \le \ell_k} \binom{x+k-j+1}{k}$.
We have
$Q_k(x) 
= \sum_{\ell_{k+1} < j \le \ell_k} \frac1{k!} (x+k-j+1) \ldots (x+1-j+1)
= \sum_{m=0}^k (\frac1{k!}  \sum_{\ell_{k+1} < j \le \ell_k} \sg^j_{k-m}) x^m$
for $k \ge 1$, and $Q_0(x) = \ell_0 - \ell_1$.
Here, $\sg^j_{k-m}$ is the symmetric product of degree $k-m$
of $k-j+1, \ldots, 1-j+1$, i.e.,
$\sg^j_{k-m} = \sum_{i_1 < \ldots < i_{k-m}} u_{i_1} \ldots u_{i_{k-m}}$
for $u_i = i-j+1 \ (1 \le i \le k)$. 
Thus, we can write as 
$Q_k(x) = \sum_{m=0}^k q_{k,m}x^m$ with 
$$
	q_{k,m} = \frac1{k!}  \sumn_{\ell_{k+1} < j \le \ell_k} \sg^j_{k-m},
$$
and in particular $q_{k,k} = (\ell_k - \ell_{k+1})/k!$.
Hence, if $\ell_{k+1}$ and $\ell_k$ can be written in terms of $p_n, \ldots, p_k$ and $n$,
then $q_{k,m} \ (0 \le m \le k)$ can also be written in terms of 
$p_n, \ldots, p_k$ and $n$.
We shall prove, by descending induction on $k$, that
$$
	\ell_n = n!p_n, \ \ 
	\ell_k = \ell_{k+1} + k!\bigg(p_k - \sumn_{j = k+1}^n q_{j,k}\bigg)
$$
for $k = n-1, \ldots, 1, 0$.

(2)
By comparing the leading terms of $P(x) = \sum_{k=0}^n Q_k(x)$,
we have $p_n = q_{n,n} = \ell_n/n!$, and thus $\ell_n = n!p_n$,
as we observed before.
At this point, as we mentioned in (1), we have explicit formulas
$q_{n,m} = \frac1{n!} \sum_{j = 1}^{n!p_n} \sg^j_{n-m}$
for $0 \le m \le n$, where $\sg^j_{n-m}$ is the symmetric product of 
degree $n-m$ of $n-j+1, \ldots, 1-j+1$.

Let us consider the next degree.
Writing $P(x) - Q_n(x) = \sum_{k=0}^{n-1} Q_k(x)$, and comparing
the leading terms, we have
$p_{n-1} - q_{n,n-1} = q_{n-1,n-1} = (\ell_{n-1} - \ell_n)/(n-1)!$.
Note that, as a consequence, $p_{n-1} - q_{n,n-1} \ge 0$
is a necessary condition for $P(x)$ to be a Hilbert polynomial.
We then have $\ell_{n-1} = \ell_n + (n-1)!(p_{n-1} - q_{n,n-1})$.
Since $\ell_n$ and $q_{n,n-1}$ are written in terms of $p_n$ and $n$ explicitly,
$\ell_{n-1}$ is written in terms of $p_n, p_{n-1}$ and $n$ explicitly.
Now by (1), $q_{n-1,m} \ (0 \le m \le n-1)$ can be written
in terms of $p_n, p_{n-1}$ and $n$ explicitly.

We can continue these processes inductively for $k = n-1, \ldots, 1, 0$,
and we have a necessary condition 
$p_k - \sum_{j = k+1}^n q_{j,k} \ge 0$ and
$\ell_k = \ell_{k+1} + k!(p_k - \sum_{j = k+1}^n q_{j,k})$ for
$k = n-1, \ldots, 1, 0$.
Thus, $\ell_k$ can be written in terms of $p_n, \ldots, p_k$ and $n$ explicitly, and hence 
$q_{k,m} \ (0 \le m \le k)$ can be written in terms of $p_n, \ldots, p_k$ and $n$ 
explicitly.
In particular, $\ell_0$ can be written in terms of $p_n, \ldots, p_0$ and $n$ explicitly.\par
We now describe how the above recursive formula leads to an explicit effective bound of $\ell_0$ in terms of $p_n, \ldots, p_0$ and $n$ as we desire. \par

(3) We fix $k \ (1 \le k \le n)$ for a while.
Recall $Q_k(x) = \sum_{m=0}^k q_{k,m}x^m = 
\sum_{m=0}^k (\frac1{k!}  \sum_{\ell_{k+1} < j \le \ell_k} \sg^j_{k-m}) x^m$,
where $\sg^j_{k-m}$ is the symmetric product of degree $k-m$ of 
$u^j_i := i-j+1 \ (1 \le i \le k)$.
Since $\ell_{k+1} < j \le \ell_k$, we see $-\ell_k \le -j \le u^j_i \le k$ 
for any $i$.
We let $\ell_k' = \max \{k, \ell_k\}$.
Then $|u^j_i| \le \ell_k'$ for any $i$, and 
$|\sg^j_{k-m}| 
\le \sum_{i_1 < \ldots < i_{k-m}} |u^j_{i_1} \ldots u^j_{i_{k-m}}|
\le \binom{k}{k-m} {\ell_k'}^{k-m}$, which is independent of $j$.
Hence 
$|q_{k,m}| \le \frac1{k!} \sum_{\ell_{k+1} < j \le \ell_k} |\sg^j_{k-m}|
	\le \frac1{k!} (\ell_k - \ell_{k+1}) \binom{k}{k-m} {\ell_k'}^{k-m}
	= \frac{\ell_k - \ell_{k+1}}{m! (k-m)!} {\ell_k'}^{k-m}$.

We will use this in the form 
$|q_{j,k}| \le \frac{\ell_j - \ell_{j+1}}{k! (j-k)!} {\ell_j'}^{j-k}$
for given $j \ (1 \le j \le n)$ and $k = j, \ldots, 1,0$.
As a consequence, we have 
$\ell_k - \ell_{k+1} = k!(p_k - \sum_{j = k+1}^n q_{j,k})
\le |k!p_k| 
	+ \sum_{j = k+1}^n \frac{\ell_j - \ell_{j+1}}{(j-k)!} {\ell_j'}^{j-k},$
which we will use in the form
$$
\ell_k - \ell_{k+1}
\le |k!p_k| + \sumn_{j = k+1}^n (\ell_j - \ell_{j+1}) {\ell_j'}^{j-k}.
$$
This holds for $k$ with $0 \le k \le n$.

(4)
We are ready to prove the effective bound.
We set $c_0 = 1, c_1 = 2, c_2 = (c_1+c_0)^3+1, \ldots,
c_k = (\sum_{k=0}^{k-1} c_j)^{k+1} + 1 \ (k = 2, \ldots, n)$.
We shall show that
(i) $c_k \le \gc_k$ for every $k \ge 0$, and
(ii) $b_{n-k} := \ell_{n-k} - \ell_{n-k+1} \le c_k \mu_P^{(k+1)!}$
for $k = n, \ldots, 1,0$.
If we have these (i) and (ii), we then have
$\ell_0 = \sum_{k=0}^n (\ell_{n-k} - \ell_{n-k+1})
\le \sum_{k=0}^n c_k \mu_P^{(k+1)!} 
\le \sum_{k=0}^n \gc_k \mu_P^{(k+1)!}$,
and we are done.

(i) 
By definition $c_0 = \gc_0, c_1 = \gc_1$.
We proceed by induction on $k \ge 2$.
Using $1 = c_0 < c_1 < \ldots$, we see $c_k \le (kc_{k-1})^{k+1}$.
Then by the induction, 
$(kc_{k-1})^{k+1} \le k^{k+1} \gc_{k-1}^{k+1} = \gc_k$.

(ii) 
This is also shown by induction on $k$.
For $k=0$, $\ell_n - \ell_{n+1} = n!p_n \le c_0\mu_P$.
We assume that our assertion holds true for up to $k-1 \ (k \ge 1)$.
Then by (3),
$b_{n-k}
\le |(n-k)!p_{n-k}| + \sum_{j = n-k+1}^n (\ell_j - \ell_{j+1}) 
												{\ell_j'}^{j-(n-k)}
\le \mu_P + \sum_{j = n-k+1}^n b_j (\max\{n, b_n+ \ldots +b_{n-k+1}\})^k
\le \mu_P + (\max\{\mu_P, b_n+ \ldots +b_{n-k+1}\})^{k+1}
\le \mu_P + (c_0\mu_P +c_1\mu_P^{2!} + \ldots + c_{k-1}\mu_P^{k!})^{k+1}$.
At the last inequality, we used the induction hypothesis.
Then $b_{n-k} 
\le \mu_P + (c_0 +c_1 + \ldots + c_{k-1})^{k+1} \mu_P^{(k+1)!}
= c_k\mu_P^{(k+1)!}$.
\end{proof}

\subsection{The bound on coefficients} \label{sec_coeff_bound}

We restate Proposition \ref{coeff_bound_1} as follows in a way that is convenient for the inductive proof.

\begin{prop}\label{coeff_bound}
Let $X$ be a canonically polarized manifold of dimension $n$,
and let $\chi(X,\CO_X(tK_X)) = \sum_{i=n,\ldots,1,0} x_i^K t^i \in \BQ[t]$
be the Hilbert polynomial. 
Then $x_n^K = K_X^n/n!$ and
$$
	|x_{n-k}^K| <  n! a_1 \cdots a_n m_n^k(1+m_n)^{nk} K_X^n 
$$
for $k = 0,1,\ldots,n$, where $m_n = 1 + \frac12 (n+1)(n+2)$ and 
$a_p = 2^{p(p+3)/2-2}/p!$ for $p \ge 1$.	
\end{prop}

\begin{proof}
We shall proceed by induction on $n$. We again denote by $v_X = K_X^n$ the canonical volume.
For $n=1$, by Riemann-Roch, we have
$\chi(X,\CO_X(tK_X)) = (2g-2)t + \chi(X,\CO_X) = v_Xt - v_X/2$,
where $g$ is the genus of $X$.
Our assertion is trivial.
We let $n \ge 2$ from now on.

(1)
Assume our assertion holds for canonically polarized manifolds of 
dimension $n-1$.
We take a canonically polarized manifold $X$ of dimension $n$.
By \cite{AS}, the complete linear system $|m_nK_X|$ is base point free
and separates any two distinct points on $X$.
Let $L_X = m_nK_X$ be a pluricanonical divisor, and take a general member 
$Y \in |L_X|$. By Bertini's theorem, $Y$ is non-singular. We set $L_Y = L_X|_Y$.
Then $K_Y = (K_X+L_X)|_Y = (1+m_n)K_X|_Y$ is ample,
and $K_Y = \frac{1+m_n}{m_n}L_Y$ (strictly speaking, these are
$\BQ$-linearly equivalent).
We let $h(tK_X) \in \BQ[t]$ (respectively $h(tL_X), h(tK_Y)$
and $h(tL_Y)$) be the Hilbert polynomial of $K_X$ (respectively
$L_X, K_Y$ and $L_Y$), and write
$$
\begin{matrix}
h(tK_X) = \sum_{i=n,\ldots,1,0} x_i^K t^i, &
h(tL_X) = \sum_{i=n,\ldots,1,0} x_i t^i, \\
h(tK_Y) = \sum_{i=n-1,\ldots,1,0} y_i^K t^i, & 
h(tL_Y) = \sum_{i=n-1,\ldots,1,0} y_i t^i. \\
\end{matrix}
$$
The relation $L_X = m_nK_X$ (resp.\ $L_Y = \frac{m_n}{1+m_n}K_Y$)
leads to relations $x_i = m_n^ix_i^K$ for $i=n,\ldots,1,0$
(resp.\ $y_i = (\frac{m_n}{1+m_n})^i y_i^K$ for $i=n-1,\ldots,1,0$).
We also have $v_Y = K_Y^{n-1} = (1+m_n)^{n-1}m_nv_X$ by
$K_Y = (1+m_n)K_X|_Y$ and $Y \in |m_nK_X|$,
and $L_Y^{n-1} = m_n^{n-1}K_X^{n-1}\cdot Y = m_n^nv_X$.
From the natural exact sequence 
$0 \lra \CO_X(-Y) \lra \CO_X \lra \CO_Y \lra 0$, we have an exact sequence
$0 \lra \CO_X((t-1)L_X) \lra \CO_X(tL_X) \lra \CO_Y(tL_Y) \lra 0$
for every integer $t$.
We then have $h(tL_X) - h((t-1)L_X) = h(tL_Y)$ as polynomials.

(2)
Since the canonical volume $v_Y$ of $Y$ is bounded by an effective
number depending only on $n$ and $v_X$, we have effective bounds of
the coefficients $y_i^K$ of $h(tK_Y)$ by the induction hypothesis.
Then by the ``effective'' relation $L_Y = \frac{m_n}{1+m_n}K_Y$,
we also have effective bounds of the coefficients $y_i$ of $h(tL_Y)$.
By the difference relation $h(tL_X) - h((t-1)L_X) = h(tL_Y)$,
we can compute $x_i$ by $y_i$ effectively, except for $x_0$.

When $t=1$, we have $h(L_X) - \chi(X,\CO_X) = h(L_Y)$ and
$x_0 = \chi(X,\CO_X) = h(L_X) - h(L_Y)$.
We have vanishing $H^q(X,\CO_X(L_X)) = H^q(X,\CO_X(m_nK_X)) = 0$ for $q > 0$, 
and $h^0(X,\CO_X(L_X)) \le L_X^n + n$ by \cite[Proposition 2.6]{Heier_Crelle}
for example.
We also have $h(L_Y) = \sum_{i=n-1,\ldots,1,0} y_i$.
Thus $|x_0| \le h^0(X,\CO_X(L_X)) + |h(L_Y)| \le L_X^n + n 
+ \sum_{i=n-1,\ldots,1,0} |y_i|$, which is effectively bounded.
The relations $x_i^K = x_i/m_n^i$ will give our effective bounds.
This is the principle for the proof.
Practically we argue as follow.

(3)
Our induction hypothesis on $y_i^K$ is that, for $k = 1,2,\ldots,n-1$,
$$
	|y_{n-1-k}^K| 
	< (n-1)! a_1 \cdots a_{n-1} v_Y m_{n-1}^k(1+m_{n-1})^{(n-1)k}. 
$$
Combining with $v_Y = (1+m_n)^{n-1}m_nv_X$, we have 
$y_{n-1} = L_Y^{n-1}/(n-1)! = m_n^n v_X/(n-1)!$ and, for $k = 1,2,\ldots,n-1$,
\begin{equation*} 
\begin{aligned}
|y_{n-1-k}| 
& = |y_{n-1-k}^K| \big(\frac{m_n}{1+m_n}\big)^{n-1-k} \\
& < (n-1)! a_1 \cdots a_{n-1} v_Y 
	m_{n-1}^k(1+m_{n-1})^{(n-1)k} \big(\frac{m_n}{1+m_n}\big)^{n-1-k} \\
& < (n-1)! a_1 \cdots a_{n-1} m_n^n v_X (1+m_n)^{nk}. 
\end{aligned}
\end{equation*}

(4)
Let us handle $x_0^K = x_0 = \chi(X,\CO_X)$ first. 
Combining our preceding observation with the induction hypothesis yields
\begin{equation*} 
\begin{aligned}
|x_0^K| 
& \le h^0(X,\CO_X(L_X)) + |h(L_Y)| \\
& \le L_X^n + n + y_{n-1} + \sumn_{k=2}^{n-1} |y_{n-1-k}| \\
& < m_n^nv_X + n + m_n^n v_X/(n-1)! 
	+ (n-1)! a_1 \cdots a_{n-1} m_n^n v_X  \sumn_{k=2}^{n-1} (1+m_n)^{nk} \\ 
& < m_n^nv_X \big(1 + n/(m_n^nv_X) + 1/(n-1)! 
	+ (n-1)! a_1 \cdots a_{n-1} \cdot n (1+m_n)^{n(n-1)} \big).
\end{aligned}
\end{equation*} 
Our claim for $|x_0^K|$ follows from 
\begin{equation*} 
\begin{aligned}
{}& 1 + n/(m_n^nv_X) + 1/(n-1)! 
	+ (n-1)! a_1 \cdots a_{n-1} \cdot n (1+m_n)^{n(n-1)} \\
& < 3 + n! a_1 \cdots a_{n-1} (1+m_n)^{n(n-1)} \\
& < 2 \cdot n! a_1 \cdots a_{n-1} (1+m_n)^{n^2} \\
& < a_n \cdot n! a_1 \cdots a_{n-1} (1+m_n)^{n^2}. 
\end{aligned}
\end{equation*} 
Note $a_1 = 1, a_2 = 4$ and $a_p > 2a_{p-1}$ for $p \ge 2$.

(5)
We now consider a general $x_i^K$.
We have 
$$
h((t-1)L_X)
= x_nt^n + \sumn_{i=n-1,\ldots,0} 
	\left(x_i + \sumn_{j=i+1}^n (-1)^{j-i}\binom{j}{i} x_j \right) t^i,
$$
and then
$
h(tL_X) - h((t-1)L_X) 
= - \sumn_{i=n-1,\ldots,0} 
	\big (\sumn_{j=i+1}^n (-1)^{j-i}\binom{j}{i} x_j \big) t^i.
$
Thus for $i = n-1,\ldots,1,0$, we have
$$
	y_i = - \sumn_{j=i+1}^n (-1)^{j-i}\binom{j}{i} x_j. 
$$
Let $U = (u_{ij})_{1 \le i,j \le n}$ be an $n \times n$ lower triangular
matrix given by $u_{ij} = (-1)^{i-j}\binom{n+1-j}{n-i}$ when $j \leq i$ and $u_{ij} = 0$ otherwise.
By letting column vectors $\vx = {}^t(x_n, \ldots, x_2,x_1)$ and
$\vy = {}^t(y_{n-1}, \ldots, y_1,y_0)$, we have $\vy = U \vx$, i.e.,
$$
\begin{pmatrix}
y_{n-1} \\
y_{n-2} \\
\vdots \\
y_0
\end{pmatrix}
=
\begin{pmatrix}
\binom{n}{n-1}&0&0&&&\\
-\binom{n}{n-2}&\binom{n-1}{n-2}&0&&&\\
\binom{n}{n-3}&-\binom{n-1}{n-3}&\binom{n-2}{n-3}&&&\\
&&\ddots&\ddots&&\\
\vdots&\vdots&&&\binom{2}{1}&0\\
(-1)^{n-1}&(-1)^{n-2}&\cdots&&-1&1
\end{pmatrix}
\begin{pmatrix}
x_n \\
x_{n-1} \\
\vdots \\
x_1 \\
\end{pmatrix}.
$$
We see $\det U = n!$, and have $\vx = U^{-1}\vy$.
Let $U^{-1} = (w_{ij})_{1 \le i,j \le n}$ be the inverse matrix of $U$,
which is lower triangular too.
We can write $w_{ij} = \frac{1}{\det U} (-1)^{j+i} \det U_{ji}$,
where $U_{ji}$ is the $(n-1) \times (n-1)$-matrix obtained from $U$
by removing the $j$-th row and the $i$-th column.
Note $|u_{ij}| = \binom{n+1-j}{n-i} < (1+1)^{n+1-j}$.
We can apply Lemma \ref{minor} below for $U_{ji}$, and we see 
$|\det U_{ji}| < 2^{n(n+3)/2-2}$, and hence 
$|w_{ij}| < 2^{n(n+3)/2-2}/n! = a_n$.

(6)
We are now ready to estimate
$x_{n-k} = \sum_{j=1}^{k+1} w_{(k+1)j} y_{n-j}$ for $k = 1,\ldots,n-1$.
By $|w_{ij}| < a_n$ in (5), we have 
$|x_{n-k}| < a_n \sum_{j=0}^{k} |y_{n-1-j}|$. 
Then for $k = 1,\ldots,n-1$, by the modified induction hypothesis (3),
\begin{equation*} 
\begin{aligned}
|x_{n-k}| 
& < a_n \sumn_{j=0}^{k} |y_{n-1-j}| \\
& < a_n (n-1)! a_1 \cdots a_{n-1} m_n^n v_X \sumn_{j=0}^k(1+m_n)^{nj} \\
& < (n-1)! a_1 \cdots a_n m_n^n v_X \cdot n (1+m_n)^{nk}.
\end{aligned}
\end{equation*} 
Then for $k = 1,\ldots,n-1$, we have
$|x_{n-k}^K| = \frac1{m_n^{n-k}}|x_{n-k}|
< n! a_1 \cdots a_n v_X m_n^k (1+m_n)^{nk}$. 
As we already know $x_n^K$ and $|x_0^K|$, this completes the proof.
\end{proof}

\begin{lem}\label{minor}
Let $n \ge 2$.
Let $V = (v_{ij})_{1 \le i,j \le n-1}$ be an $(n-1) \times (n-1)$-matrix 
satisfying (i) $v_{ij} = 0$ if $i+1 < j$, and 
(ii) $|v_{ij}| < 2^{n+1-j}$ for every $i,j$.
Then $|\det V| < 2^{n(n+3)/2-2}$.
\end{lem}

\begin{proof}
Let $S_{n-1}$ be the group of permutations among $\{1,2,\ldots,n-1\}$.
For every $\sg \in S_{n-1}$, we see 
$|v_{\sg(1)1} v_{\sg(2)2} \cdots v_{\sg(n-1) n-1}| 
< 2^n 2^{n-1} \cdots 2^2 = 2^{n(n+1)/2-1}$.
Let $S_{n-1}^V = \{\sg \in S_{n-1} ;\
  v_{\sg(1)1} v_{\sg(2)2} \cdots v_{\sg(n-1) n-1} \ne 0 \}
  = \{\sg \in S_{n-1} ;\
  v_{1\sg(1)} v_{2\sg(2)} \cdots v_{n-1 \sg(n-1)} \ne 0 \}$.
We see the number of elements of $S_{n-1}^V$ is not greater than 
$2^{n-1}$ due to the shape of $V$ as given in (i).
Thus 
$|\det V| = |\sum_{\sg \in S_{n-1}^V} 
\text{sgn}(\sg) v_{\sg(1)1} v_{\sg(2)2} \cdots v_{\sg(n-1) n-1} |
< 2^{n-1} 2^{n(n+1)/2-1}
= 2^{n(n+3)/2-2}$.
\end{proof}

\end{document}